\documentclass{article}
\usepackage[latin1]{inputenc} 
\usepackage[T1]{fontenc} 
\usepackage[english]{babel} 
\usepackage{indentfirst} 
\usepackage{latexsym}
\usepackage{amssymb}
\usepackage{amsmath}
\usepackage{amsthm}
\usepackage{url}
\usepackage{enumitem}

\newtheorem{thm}{Theorem}[section]
\newtheorem{cor}[thm]{Corollary}
\newtheorem{lem}[thm]{Lemma}
\newtheorem{prop}[thm]{Proposition}
\newtheorem{defn}[thm]{Definition}

\newtheorem*{thm*}{Theorem}

\newcommand{\U}{\mathcal{U}}
\newcommand{\N}{\mathbb{N}}

\newcommand{\V}{\mathcal{V}}
\newcommand{\W}{\mathcal{W}}
\newcommand{\bN}{\beta\mathbb{N}}

\newcommand{\fe}{\leq_{fe}}

\begin{document}

\title{Ultrafilters maximal for finite embeddability}

\author{Lorenzo Luperi Baglini\thanks{University of Vienna, Faculty of Mathematics, Oskar-Morgenstern-Platz 1, 1090 Vienna, AUSTRIA, e-mail: \texttt{lorenzo.luperi.baglini@univie.ac.at}, supported by grant P25311-N25 of the Austrian Science Fund FWF.}}
\maketitle
\date{}

\begin{abstract}

In \cite{fe} the authors showed some basic properties of a pre-order that arose in combinatorial number theory, namely the finite embeddability between sets of natural numbers, and they presented its generalization to ultrafilters, which is related to the algebraical and topological structure of the Stone-\v{C}ech compactification of the discrete space of natural numbers. In this present paper we continue the study of these pre-orders. In particular, we prove that there exist ultrafilters maximal for finite embeddability, and we show that the set of such ultrafilters is the closure of the minimal bilateral ideal in the semigroup $(\bN,\oplus)$, namely $\overline{K(\bN,\oplus)}$. As a consequence, we easily derive many combinatorial properties of ultrafilters in $\overline{K(\bN,\oplus)}$. We also give an alternative proof of our main result based on nonstandard models of arithmetic.

\end{abstract}

\section{Introduction}

This paper is a planned sequel of the paper \cite{fe} written by Andreas Blass and Mauro Di Nasso. Both in \cite{fe} and in this present paper it is studied a notion that arose in combinatorial number theory (see \cite{DN} and \cite{ruzsa}, where this notion was implicitly used), the finite embeddability between sets of natural numbers. We recall its definition:

\begin{defn}[\cite{fe}, Definition 1] For $A,B$ subsets of $\N$, we say that $A$ is finitely embeddable in $B$ and we write $A\fe B$ if each finite subset $F$ of $A$ has a rightward translate $F+k$ included in $B$. \end{defn}

We use the standard notation $n+F=\{n+a\mid a\in F\}$ and we use the standard convention that $\N=\{0,1,2,...\}$. In \cite{fe} the authors also considered the generalization of $\leq_{fe}$ to ultrafilters:

\begin{defn}[\cite{fe}, Definition 2] For ultrafilters $\U, \V$ on $\N$, we say that $\U$ is finitely embeddable in $\V$ and we write $\U\fe\V$ if, for each set $B\in\V$, there is some $A\in\U$ such that $A\fe B$. \end{defn}

It is easy to prove (see \cite{fe}, \cite{Tesi}) that both $(\mathcal{P}(\N),\fe)$ and $(\bN,\fe)$ are preorders. In \cite{fe} the authors studied some properties of $\leq_{fe}$, giving in particular many equivalent characterization of the relations $A\fe B$ and $\U\fe\V$ using standard and nonstandard techniques; in this present paper we use similar techniques to continue the study of these pre-orders. Our main result is that there exist ultrafilters maximal for finite embeddability and that the set of such maximal ultrafilters is the closure of the minimal bilateral ideal in $(\bN,\oplus)$, namely $\overline{K(\bN,\oplus)}$. This result allows to easily deduce many combinatorial properties of ultrafilters in $\overline{K(\bN,\oplus)}$, e.g. that for every ultrafilter $\U\in\overline{K(\bN,\oplus)}$, for every $A\in\U$, $A$ has positive upper Banach density, it contains arbitrarily long arithmetic progressions and it is piecewise syndetic\footnote{Let us note that many of these combinatorial properties of ultrafilters in $\overline{K(\bN,\oplus)}$ where already known.}. We will also show that there do not exist minimal sets in $(\mathcal{P}_{\aleph_{0}}(\N),\fe)$ or minimal ultrafilters in $(\bN\setminus\N,\fe)$, where $\mathcal{P}_{\aleph_{0}}(\N)$ is the set of infinite subsets of $\N$ and $\bN\setminus\N$ is the set of nonprincipal ultrafilters. These topics are studied in sections \ref{propy} and \ref{extreme}. In section \ref{NS23} we reprove our main result by nonstandard methods; nevertheless, this is the only section in which nonstandard methods are used, so the rest of the paper is accessible also to readers unfamiliar with nonstandard methods.\par
We refer to \cite{rif12} for all the notions about combinatorics and ultrafilters that we will use, to \cite{rif5}, §4.4 for the foundational aspects of nonstandard analysis and to \cite{davis} for all the nonstandard notions and definitions. Finally, we refer the interested reader to \cite{Tesi}, Chapter 4 for other properties and characterizations of the finite embeddability.

\section{Some basic properties of $(\mathcal{P}(\N),\fe)$}\label{propy}

Let $n$ be a natural number. Throughout this section we will denote by $\mathcal{P}_{\geq n}(\N)$ the set

\begin{equation*} \mathcal{P}_{\geq n}(\N)=\{A\subseteq\N\mid |A|\geq n\}; \end{equation*}

similarly, we will denote by $\mathcal{P}_{\aleph_{0}}(\N)$ the set

\begin{equation*} \mathcal{P}_{\aleph_{0}}(\N)=\{A\subseteq\N\mid |A|=\aleph_{0}\}. \end{equation*}

Moreover, we will denote by $\equiv_{fe}$ the equivalence relation such that, for every $A,B\subseteq\N$,

\begin{equation*} A\equiv_{fe} B\Leftrightarrow A\fe B \wedge B\fe A \end{equation*}

and, for every set $A$, we will denote by $[A]$ its equivalence class. Finally we will denote by $\leq_{fe}$ the ordering induced on the space of equivalence classes defined by setting, for every $A,B\subseteq\N$, 

\begin{equation*} [A]\fe [B]\Leftrightarrow A\fe B.\end{equation*}

It is immediate to see that the relation $\leq_{fe}$ on $\mathcal{P}(\N)$ is not antysimmetric (e.g., $\{2n\mid n\in\N\}\equiv_{fe}\{2n+1\mid n\in\N\}$), so to search for maximal and minimal sets we will actually work in $(\mathcal{P}(\N)/\mathord\equiv_{fe},\fe)$.\par
In \cite{fe} the authors proved that the finite embeddability has the following properties (for the relevant definitions, see \cite{rif12}):

\begin{prop}[\cite{fe}, Proposition 6]\label{trs} Let $A,B$ be sets of natural numbers. 
\begin{enumerate}
[leftmargin=*,label=(\roman*),align=left ]
  \item $A$ is maximal with respect to $\fe$ if and only if it is thick;
	\item if $A\fe B$ and $A$ is piecewise syndetic then $B$ is also piecewise syndetc;
	\item if $A\fe B$ and $A$ contains a $k$-term arithmetic progression then also $B$ contains a $k$-term arithmetic progression;
	\item if $A\fe B$ then the upper Banach densities satisfy $BD(A)\leq BD(B)$;
	\item if $A\fe B$ then $A-A\subseteq B-B$;
	\item if $A\fe B$ then $\bigcap\limits_{t\in G} (A-t)\fe \bigcap\limits_{t\in G}(B-t)$ for every finite $G\subseteq\mathbb{N}$.
\end{enumerate}
\end{prop}

We will use Proposition \ref{trs} to (re)prove some combinatorial properties of ultrafilters in $\overline{K(\bN,\oplus)}$ in Section \ref{extreme}. In this present section we want to study the existence of minimal elements with respect to $\fe$ in various subsets of $\mathcal{P}(\N)$, and a nice property of the ordering $\fe$ on the set of equivalence classes, namely that for every set $A$ there does not exist a set $B$ such that $[A]<_{fe} [B] <_{fe} [A+1]$. To prove this result we need the following lemma:

\begin{lem}\label{basico} For every $A,B\subseteq\N$ the following two properties hold:
\begin{enumerate}
[leftmargin=*,label=(\roman*),align=left ]
	\item\label{a1} if $B\nleq_{fe} A$ and $B\fe A+1$ then $B\subseteq A+1$;
	\item\label{a2} if $A\leq B$ and $A+1\nleq B$ then $A\subseteq B$.
\end{enumerate}\end{lem}

\begin{proof} We prove only \ref{a1}, since \ref{a2} can be proved similarly. Let $F\subseteq B$ be a finite subset of $B$ such that $F+n\nsubseteq A$ for every $n\in\N$. In particular, for every finite $H\subseteq B$ such that $F\subseteq H$ and for every $n\in\N$ we have that $n+H\nsubseteq A$. But, by hypothesis, there exists $n\in\N$ such that $n+H\subseteq A+1$. If $n\geq 1$ we have a contradition, so it must be $n=0$, i.e $H\subseteq A+1$. Since this holds for every finite $H\subseteq B$ (with $F\subseteq H$) we deduce that $B\subseteq A+1$. \end{proof}
 
\begin{thm}\label{ledzeppelin} Let $A,B\subseteq\N$. If $A\fe B\fe A+1$ then $[A]=[B]$ or $[A+1]=[B]$. \end{thm}

\begin{proof} Let us suppose that $A+1\nleq_{fe} B\nleq_{fe} A$. Then, since $A\fe B\fe A+1$, by Lemma \ref{basico} we deduce that $A\subseteq B\subseteq A+1$, so $A\subseteq A+1$. This is absurd since $A\setminus (A+1)\supseteq\{\min A\}\neq\emptyset$. \end{proof}

We now turn the attention to the existence of minimal elements in various subsets of $\mathcal{P}(\N)$. Two immediate observations are that the empty set is the minimum in $(\mathcal{P}(\N),\fe)$ and that $\{0\}$ is the minimum in $(\mathcal{P}(\N)_{\geq 1}\equiv_{fe},\fe)$. Moreover, if we identify each natural number $n$ with the singleton $\{n\}$, it is immediate to see that $(\N,\leq)$ forms an initial segment of $(\mathcal{P}_{\geq 1}(\N),\fe)$ and that, more in general, the following easy result holds:

\begin{prop} A set $A$ is minimal in $(\mathcal{P}_{\geq n}(\N),\fe)$ if and only if $0\in A$ and $|A|=n$. \end{prop}

The proof follows easily from the definitions. Let us note that, in particular, the following facts follow:
\begin{enumerate}
[leftmargin=*,label=(\roman*),align=left ]
	\item for every natural number $m\geq n-1$ there are $\binom{m}{n-1}$ inequivalent minimal elements in $(\mathcal{P}_{\geq n}(\N),\leq_{fe})$ that are subsets of $\{0,...,m\}$;
	\item if $n\geq 2$ then $(\mathcal{P}_{\geq n}(\N),\leq_{fe})$ does not have a minimum element.

\end{enumerate}

If we consider only infinite subsets of $\N$ the situation is different: there are no minimal elements in $(\mathcal{P}_{\aleph_{0}}(\N)/\mathord\equiv_{fe},\fe)$, as we are now going to show.

\begin{defn} Let $A,B\subseteq\N$. We say that $A$ is strongly non f.e. in $B$ (notation: $A\nleq_{fe}^{S} B$) if for every set $C\subseteq A$ with $|C|= 2$ we have that $C\nleq_{fe}B$. If both $A\nleq_{fe}^{S} B$ and $B\nleq_{fe}^{S} A$ we say that $A,B$ are strongly mutually unembeddable (notation: $A\not\equiv_{S} B$). \end{defn}

Let us observe that, in the previous definition, we can equivalenty substitute the condition "$|C|=2$" with "$|C|\geq 2$".

\begin{prop}\label{incomparabili1} Let $X$ be an infinite subset of $\N$. Then there are $A,B\subseteq X$, $A,B$ infinite, such that $A\cap B=\emptyset$ and $A\not\equiv_{S} B$. \end{prop}

\begin{proof} To prove the thesis we construct $A,B\subseteq X$ such that, for any $C\subseteq A$, $D\subseteq B$ with $|C|=|D|=2$, we have $C\nleq_{fe} B$ and $D\nleq_{fe} B$.\par
Let $X=\{x_{n}\mid n\in\N\}$, with $x_{n}<x_{n+1}$ for every $n\in\N$. We set

\begin{equation*} a_{0}=x_{0}, b_{0}=x_{1} \end{equation*}

and, recursively, we set

\begin{equation*} a_{n+1}=\min\{x\in X\mid x>a_{n}+b_{n}+1\}, \ b_{n+1}=\min\{x\in X\mid x>b_{n}+a_{n+1}+1\}. \end{equation*}

Finally, we set $A=\{a_{n}\mid n\in\N\}$ and $B=\{b_{n}\mid n\in\N\}$. Clearly $A\cap B=\emptyset$, and both $A,B$ are infinite subsets of $X$. Now we let $a_{n_{1}}<a_{n_{2}}$ be any elements in $A$. Let us suppose that there are $b_{m_{1}}<b_{m_{2}}$ in $B$ with $a_{n_{2}}-a_{n_{1}}=b_{m_{2}}-b_{m_{1}}$ and let us assume that $b_{n_{2}}>a_{n_{2}}$ (if the converse hold, we can just exchange the roles of $a_{n_{1}},a_{n_{2}},b_{m_{1}},b_{m_{2}}$). By construction, since $b_{m_{2}}>a_{n_{2}}$, we have $b_{m_{2}}-b_{m_{1}}\geq a_{n_{2}}+1 >a_{n_{2}}$, while $a_{n_{2}}-a_{n_{1}}\leq a_{n_{2}}$. So $A\not\equiv_{S} B$.\end{proof}

Three corollaries follow immediatly by Proposition \ref{incomparabili1}:

\begin{cor}\label{nitrogeno} For every infinite set $X\subseteq\N$ there is an infinite set $A\subseteq X$ such that $X\nleq_{fe} A$. \end{cor}
\begin{proof} Let $A,B$ be infinite subsets of $X$ such that $A\not\equiv_{S} B$. Then $X$ cannot be finitely embeddable in both $A$ and $B$ otherwise, since clearly $A,B\fe X$, we would have that $[A]=[X]=[B]$, which is absurd. \end{proof} 

\begin{cor}\label{popporoppo} For every infinite set $X\subseteq\N$ there is an infinite descending chain $X=X_{0}\supset X_{1}\supset X_{2}...$ in $\mathcal{P}_{\aleph_{0}}(\N)$ such that $X_{i+1}\nleq_{fe} X_{i}$ for every $i\in\mathbb{N}$. \end{cor}
\begin{proof} The result follows immediatly by Corollary \ref{nitrogeno}.\end{proof}

\begin{cor}\label{gugugaga} There are no minimal elements in $(\mathcal{P}_{\aleph_{0}}(\N)/\mathord\equiv_{fe},\fe)$. \end{cor}

\begin{proof} The result follows immediatly by Corollary \ref{popporoppo}. \end{proof}

\section{Properties of $(\bN,\fe)$}\label{extreme}

In this section we want to prove some basic properties of $(\bN,\fe)$, in particular the generalization of Theorem \ref{ledzeppelin} to ultrafilters, and to characterize the maximal ultrafilters with respect to $\fe$. We fix some notations: we will denote by $\equiv_{fe}$ the equivalence relation such that, for every $\U,\V$ ultrafilters on $\N$, 

\begin{equation*} \U\equiv_{fe} \V\Leftrightarrow \U\fe \V \wedge \U\fe \V \end{equation*}

and, for every ultrafilter $\U$, we will denote by $[\U]$ its equivalence class. Finally we will denote by $\leq_{fe}$ the ordering induced on the space of equivalence classes defined by setting, for every $\U,\V\in\bN$, 

\begin{equation*} [\U]\fe [\V]\Leftrightarrow \U\fe \V.\end{equation*}

\subsection{Some basic properties of $(\bN,\fe)$}

The first result that we prove is that Theorem \ref{ledzeppelin} can be generalized to ultrafilters:

\begin{thm} For every $\U,\V\in\bN$ if $\U\fe \V\fe \U\oplus 1$ then $[\U]=[\V]$ or $[\U\oplus 1]=[\V]$. \end{thm}

\begin{proof} Let us suppose that $\U\oplus 1\nleq_{fe} \V\nleq_{fe} \U$. In particular, $\U\oplus 1\neq \V$, so there exists $A\in\U$ such that $A+1\notin \V$. Since $\V\nleq_{fe} \U$ there exists $B\in\U$ such that $K\nleq_{fe} B$ for every $K\in\V$. In particular, $K\nleq A\cap B$ for every $K\in\V$.\par
Moreover, since $(A\cap B)+1\in\U\oplus 1$ we derive that there exists $C\in\V$ such that $C\fe (A\cap B)+1$. So we have that

\begin{equation*} C\nleq_{fe} (A\cap B) \ \mbox{and} \ C\fe (A\cap B)+1; \end{equation*}

by Lemma \ref{basico} we conclude that $C\subseteq (A\cap B)+1$. But $C\in\V$, so $(A\cap B)+1\in\V$ and, since $(A\cap B)+1\subseteq A+1$, this entails that $A+1\in\V$, which is absurd. \end{proof}

Another result that we want to prove is that $(\bN,\fe)$ is not a total preorder:

\begin{prop} There are nonprincipal ultrafilters $\U,\V$ such that $\U$ is not finitely embeddable in $\V$ and $\V$ is not finitely embeddable in $\U$. \end{prop}

\begin{proof} Let $A,B$ be strongly mutually unembeddable infinite sets (which existence is a consequence of Proposition \ref{incomparabili1}). Let $\U,\V$ be nonprincipal ultrafilters such that $A\in\U, B\in\V$ and let us suppose that $\U\fe\V$. Let $C\in \U$ be such that $C\fe B$. Since $C\in\U$, $A\cap C$ is in $\U$ and it is infinite (since $\U$ is nonprincipal). So we have that

\begin{itemize}
	\item $A\cap C\fe B$, since $A\cap C\subseteq C$;
	\item $A\cap C\nleq_{fe} B$, since $A\not\equiv_{S} B$.
\end{itemize}

This is absurd, so $\U$ is not finitely embeddable in $\V$. In the same way we can prove that $\V$ is not finitely embeddable in $\U$.\end{proof}

It is easy to show that, if we identity each natural number $n$ with the principal ultrafilter $\U_{n}=\{A\in\mathcal{P}(\N)\mid n\in A\}$, then $(\N,\leq)$ is an initial segment in $(\bN,\fe)$. In particular, $\U_{0}$ is the minimum element in $\bN$. One may wonder if there is a minimum element in $(\bN\setminus\N,\fe),$ and the answer is no. In the following proposition, by $\Theta_{X}$ we mean the clopen set

\begin{equation*} \Theta_{X}=\{\U\in\bN\mid X\in\U\}. \end{equation*}

\begin{prop} For every infinite set $X\subseteq\N$ there is not a minimum in $((\Theta_{X}\setminus\N)/\mathord\equiv_{fe},\fe)$. \end{prop}

\begin{proof} Let us suppose that such a minimum $M$ exists, and let $\U\in\Theta_{X}$ be such that $M=[\U]$. Let $A,B\subseteq X$ be mutually unembeddable subsets of $X$ and let $\V_{1},\V_{2}$ be nonprincipal ultrafilters such that $A\in \V_{1}$ and $B\in \V_{2}$ (in particular, $\V_{1},\V_{2}\in\Theta_{X}$). Since, by hypothesis, $[\U]$ is the minumum in $((\Theta_{X}\setminus\N)/\mathord\equiv_{fe},\fe)$, there are $C_{1},C_{2}\in \U$ such that $C_{1}\fe A$ and $C_{2}\fe B$. Let us consider $C_{1}\cap C_{2}\in\U$. By construction, $C_{1}\cap C_{2}$ is finitely embeddable in $A$ and in $B$. But this is absurd: in fact, let $c_{1}<c_{2}$ be any two elements in $C_{1}\cap C_{2}$. Then there are $n,m$ such that $n+\{c_{1},c_{2}\}=\{a_{1},a_{2}\}\subset A$ and $m+\{c_{1},c_{2}\}=\{b_{1},b_{2}\}\subset B$, and this cannot happen, because in this case we would have $b_{2}-b_{1}=c_{2}-c_{1}=a_{2}-a_{1}$, while $A\not\equiv_{S} B$. \end{proof}

In particular, by taking $X=\N$, we prove that:

\begin{cor} There is not a minimum in $((\bN\setminus\N)/\mathord\equiv_{fe},\fe)$.\end{cor}

\subsection{Maximal Ultrafilters}

To study maximal ultrafilters in $(\bN,\fe)$ we need to recall three results that have been proved in \cite{fe}:

\begin{thm}[\cite{fe}, Theorem 10]\label{fondamentali} Let $\U,\V$ be ultrafilters on $\N$. Then $\U\fe\V$ if and only if $\V\in\overline{\{\U\oplus\W\mid \W\in\bN\}}$. \end{thm}

\begin{cor}[\cite{fe}, Corollary 12]\label{updir} The ordering $\fe$ on ultrafilters on $\N$ is upward directed.\end{cor}

We also recall that, actually, Corollary \ref{updir} can be improved: in fact, for every $\U,\V\in\bN$ we have

\begin{equation*} \U,\V\fe\U\oplus\V. \end{equation*}

Let us introduce the following definition:

\begin{defn} For any $\U\in\bN$ the upward cone generated by $\U$ is the set

\begin{equation*} \mathcal{C}(\U)=\{\V\in\bN\mid \U\fe\V\}. \end{equation*}

\end{defn} 

\begin{cor}[\cite{fe}, Corollary 13]\label{hui} For any $\U\in\bN$, the upward cone $\mathcal{C}(\U)$ is a closed, two-sided ideal in $\bN$. It is the smallest closed right ideal containing $\U$ and therefore it is also the smallest two-sided ideal containing $\U$.\end{cor}

Let us note that from Theorem \ref{fondamentali} it easily follows that the relation $\fe$ is not antisymmetric: in fact, if $R$ is a minimal right ideal in $(\bN,\oplus)$ and $\U\in R$ then $\mathcal{C}(\U)=\mathcal{C}(\U\oplus 1)$, so $\U\fe\U\oplus 1$ and $\U\oplus 1\fe \U$.\par
We want to prove that there is a maximum in $(\bN/\mathord\equiv_{fe},\fe)$. Due to Corollary \ref{hui}, since $(\bN/\mathord\equiv_{fe},\fe)$ is an order then to prove that it has a maximum if is enough\footnote{An upward directed ordered set $(A,\leq)$ has at most one maximal element which, if it exists, is the greatest element of the order.} to prove that it has maximal elements.\par
To prove the existence of maximal elements we use Zorn's Lemma. A technical lemma that we need is the following:

\begin{lem}\label{ultraordine} Let $I$ be a totally ordered set. Then there is an ultrafilter $\V$ on $I$ such that, for every element $i\in I$, the set 

\begin{equation*} G_{i}=\{j\in I\mid j\geq i\}.\end{equation*}

is included in $\V$.

\end{lem}

\begin{proof} We have just to observe that $\{G_{i}\}_{i\in I}$ is a filter and to recall that every filter can be extended to an ultrafilter.\end{proof}

The key property of these ultrafilters is the following:

\begin{prop}\label{ultraordine2} Let $I$ be a totally ordered set and let $\V$ be given as in Lemma \ref{ultraordine}. Then for every $A\in \V$ and $i\in I$ there exists $j\in A$ such that $i\leq j$. \end{prop}

We omit the straightforward proof.\par
In the next Theorem we use the notion of limit ultrafilter. We recall that, given an ordered set $I$, an ultrafilter $\V$ on $I$ and a family $\U_{i}$ of ultrafilters on $\N$, the $\V$-limit of the family $\langle \U_{i}\mid i\in I\rangle$ (denoted by $\V-\lim\limits_{i\in I}\U_{i}$) is the ultrafilter such that, for every $A\subseteq\N$, 

\begin{equation*} A\in\V-\lim\limits_{i\in I}\U_{i}\Leftrightarrow\{i\in I\mid A\in\U_{i}\}\in\V. \end{equation*}

Let us introduce the notion of $\fe$-chain:

\begin{defn} Let $(I,<)$ be an ordered set. We say that the family $\langle \U_{i}\mid i\in I\rangle$ is an $\fe$-chain if for every $i<j\in I$ we have $\U_{i}\fe\U_{j}$.\end{defn}

\begin{thm} Every $\fe$-chain $\langle \U_{i}\mid i\in I\rangle$ has an $\fe$-upper bound $\U$. \end{thm}

\begin{proof} Let $\V$ be an ultrafilter on $I$ with the property expressed in Lemma \ref{ultraordine}. We claim that the ultrafilter 

\begin{equation*} \U=\V-\lim\limits_{i\in I} \U_{i} \end{equation*}

is an $\fe$-upper bound for the $\fe$-chain $\langle \U_{i}\mid i\in I\rangle$. We have to prove that $\U_{i}\fe \U$ for every index $i$; let $A$ be an element of $\U$. By definition,

\begin{equation*} A\in \U\Leftrightarrow I_{A}=\{i\in I\mid A\in \U_{i}\} \in \V. \end{equation*}

$I_{A}$ is a set in $\V$ so, by Proposition \ref{ultraordine2}, there is an element $j>i$ in $I_{A}$. Therefore $A\in\U_{j}$ and, since $\U_{i}\fe \U_{j}$, there exists an element $B$ in $\U_{i}$ with $B\fe A$. Hence $\U_{i}\fe\U$, and the thesis is proved.\end{proof}

As an immediate consequence we have that:

\begin{cor} Every $\fe$-chain $\langle [\U_{i}]\mid i\in I\rangle$ has an upper bound $[\U]$. \end{cor}

Being an upward directed set with maximal elements, $(\bN/\mathord{\equiv_{fe}},\fe)$ has a maximum, that we denote by $M$.

\begin{defn} We say that an ultrafilter $\U$ on $\N$ is maximal if $[\U]=M$. We denote by $\mathcal{M}$ the set of maximal ultrafilters.\end{defn}

By definition, for every ultrafilter $\U$ we have the following equivalences: 

\begin{equation*} [\U]=M\Leftrightarrow \U\in\mathcal{M}\Leftrightarrow \V\fe\U \ \mbox{for every} \ \V\in\bN. \end{equation*}

In particular, we can characterize $\mathcal{M}$ in terms of the $\fe$-cones:

\begin{cor}\label{zumpa} $\mathcal{M}=\bigcap\limits_{\U\in\bN} \mathcal{C}(\U).$\end{cor}

\begin{proof} We have just to observe that $\mathcal{M}\subseteq\mathcal{C}(\U)$ for every ultrafilter $\U$ and that, if $\U$ is a maximal ultrafilter, then $\mathcal{C}(\U)=\mathcal{M}$.\end{proof}

We can now prove our main result:

\begin{thm}\label{eccolo} $\mathcal{M}=\overline{K(\bN,\oplus)}$.
\end{thm}

\begin{proof} Given any ultrafilter $\U$, by Proposition \ref{fondamentali} we know that $\mathcal{C}(\U)$ is the minimal closed bilateral ideal containing $\U$. By Corollary $\ref{zumpa}$ we know that $\mathcal{M}=\bigcap\limits_{\U\in\bN}\mathcal{C}(\U)$ so, in particular, being the intersection of a family of closed bilateral ideal $\mathcal{M}$ itself is a closed bilater ideal. So if $\U$ is any ultrafilter in $K(\bN,\oplus)$, we know that:
\begin{enumerate}
[leftmargin=*,label=(\roman*),align=left ]
	\item $\mathcal{M}\subseteq\mathcal{C}(\U)$;
	\item $\mathcal{C}(\U)=\overline{K(\bN,\oplus)}$.
\end{enumerate}

So $\mathcal{M}$ is a closed bilateral ideal included in $\overline{K(\bN,\oplus)}$, and the only such ideal is $\overline{K(\bN,\oplus)}$ itself.\end{proof}
This result has a few interesting consequences:

\begin{cor}\label{uno} An ultrafilter $\U$ is maximal if and only if every element $A$ of $\U$ is piecewise syndetic. \end{cor}

\begin{proof} This follows from this well-known characterization of $\overline{K(\bN,\oplus)}$: an ultrafilter $\U$ is in $\overline{K(\bN,\oplus)}$ if and only if every element $A$ of $\U$ is piecewise syndetic (see, e.g., \cite{rif12}). \end{proof}

As mentioned in the introduction, the notion of finite embeddability is related with some properties that arose in combinatorial number theory. A particularity of maximal ultrafilters is that every set in a maximal ultrafilter satisfies many of these combinatorial properties:

\begin{defn} We say that a property $P$ is $\leq_{fe}$-upward invariant if the following holds: for every $A,B\subseteq \N$, if $P(A)$ holds and $A\leq_{fe} B$ then $P(B)$ holds.\par
We way that $P$ is partition regular if the family $S_{P}=\{A\subseteq\N\mid P(A)$ holds$\}$ contains an ultrafilter (i.e., if for every finite partition $\N=A_{1}\cup... \cup A_{n}$ there exists at least one index $i\leq n$ such that $A_{i}\in S_{P})$.\end{defn}

By Proposition \ref{trs} it follows that the following properties are $\leq_{fe}$-upward invariant:

\begin{enumerate}
[leftmargin=*,label=(\roman*),align=left ]
	\item $A$ is thick;
	\item\label{ps} $A$ is piecewyse syndetic;
	\item\label{alap} $A$ contains arbitrarily long arithmetic progressions;
	\item\label{bd} $BD(A)>0$, where $BD(A)$ is the upper Banach density of $A$.
\end{enumerate}
In particular, properties \ref{ps}, \ref{alap}, \ref{bd} are also partition regular. These kind of properties are important in relation with maximal ultrafilters:

\begin{prop}\label{consequence} Let $P$ be a partition regular $\leq_{fe}$-upward invariant property of sets. Then for every maximal ultrafilter $\U$, for every $A\in\U$, $P(A)$ holds.\end{prop}

\begin{proof} Let $P$ be given, let $S_{P}=\{A\subseteq\N\mid P(A)$ holds$\}$ and let $\V\subseteq S_{P}$ (such an ultrafilter exists because $P$ is partition regular). Let $B\in\U$. Since $\U$ is maximal, $\V\fe\U$. Let $A\in\V$ be such that $A\fe B$. Since $P$ is $\fe$-upward invariant and $P(A)$ holds, we obtain that $P(B)$ holds, hence we have the thesis. \end{proof}

E.g., as a consequence of Proposition \ref{consequence} we can prove the following:

\begin{cor}\label{due} Let $\U\in\overline{K(\bN,\oplus)}$. Then:

\begin{enumerate}
[leftmargin=*,label=(\roman*),align=left ]

\item each set $A$ in $\U$ has positive Banach density;
\item each set $A$ in $\U$ contains arbitrarily long arithmetic progressions;
\item each set $A$ in $\U$ is piecewise syndetic.
\end{enumerate}

\end{cor}

In particular, by combining Corollaries \ref{uno} and \ref{due} we obtain an alternative proof of the following known results:
\begin{itemize}
	\item every piecewise syndetic set contains arbitrarily long arithmetic progressions;
	\item every piecewise syndetic set has positive upper Banach density.
\end{itemize}

In the forthcoming paper \cite{functions} we will show how, actually, similar arguments can be used to prove combinatorial properties of other families of ultrafilters, e.g. to prove that for every ultrafilter $\U\in\overline{K(\bN,\odot)}$, for every $A\in\U$, $A$ contains arbitrarily long arithmetic progression and it contains a solution to every partition regular homogeneous equation\footnote{An equation $P(x_{1},...,x_{n})=0$ is partition regular if and only if for every finite coloration $\N=C_{1}\cup...\cup C_{n}$ of $\N$ there exists an index $i$ and monocromatic elements $a_{1},...,a_{n}\in C_{i}$ such that $P(a_{1},...,a_{n})=0$.}.

\section{A Direct Nonstandard Proof that $M_{fe}=\overline{K(\bN,\oplus)}$}\label{NS23}

In this section we assume the reader to be familiar with the basics of nonstandard analysis. In particular, we will use the notions of nonstandard extension of subsets of $\N$ and the transfer principle. We refer to \cite{rif5} and \cite{davis} for an introduction to the foundations of nonstandard analysis and to the nonstandard tools that we are going to use.\par
Both in \cite{fe} and in \cite{Tesi} it has been shown that the relation of finite embeddability between sets has a very nice characterization in terms of nonstandard analysis, which allows to study some of its properties in a quite simple, and elegant, way. We recall the characterization (in the following proposition, it is assumed for technical reasons that the nonstandard extension that we consider satisfies at least the $\mathfrak{c}^{+}$-enlarging property\footnote{We recall that a nonstandard extension $^{*}\N$ of $\N$ has the $\mathfrak{c}^{+}$ enlarging property if, for every family $\mathcal{F}$ of subsets of $\N$ with the finite intersection property, the intersection $\bigcap\limits_{A\in\mathcal{F}}$$^{*}A$ is nonempty.}, where $\mathfrak{c}$ is the cardinality of $\mathcal{P}(\N)$):

\begin{prop}[\cite{fe}, Proposition 15]\label{NSCAR} Let $A,B$ be subsets of $\N$. The following two conditions are equivalent:
\begin{enumerate}
[leftmargin=*,label=(\roman*),align=left ]
	\item $A$ is finitely embeddable in $B$;
	\item there is an hypernatural number $\alpha$ in $^{*}\N$ such that $\alpha+A\subseteq$$^{*}B$.
\end{enumerate}
\end{prop}

We use Proposition \ref{NSCAR} to reprove directly, with nonstandard methods, Theorem \ref{eccolo}: 

\begin{proof}[Theorem \ref{eccolo}] Let $A$ be a set in $\U$, and let $\V$ be an ultrafilter on $\N$. Since $A$ is piecewise syndetic there is a natural number $n$ such that 

\begin{equation*} T=\bigcup_{i=1}^{n} (A+i) \end{equation*} 

is thick. By transfer\footnote{Thick set can be characterized by mean of nonstandard analysis as follows (see e.g. \cite{Tesi}): a set $T\subseteq\N$ is thick if and only if $T^{*}$ contains an interval of infinite lenght.} it follows that there are hypernatural numbers $\alpha\in$$^{*}\N$ and $\eta\in$$^{*}\N\setminus\N$ such that the interval $[\alpha,\alpha+\eta]$ is included in $^{*}T$. In particular, since $\eta$ is infinite, $\alpha+\N\subseteq$$^{*}T$.\par
For every $i\leq n$ we consider 

\begin{equation*} B_{i}=\{n\in\N\mid \alpha+n\in\mbox{}^{*}(A+i)\}. \end{equation*}

Since $\bigcup_{i=1}^{n} B_{i}=\N$, there is an index $i$ such that $B_{i}\in\V$. We claim that $B_{i}\fe A$. In fact, by construction $\alpha+B_{i}\subseteq$$^{*}A+i$, so

\begin{equation*} (\alpha-i)+B_{i}\subseteq\mbox{}^{*}A. \end{equation*} 
By Proposition \ref{NSCAR}, this entails that $B_{i}\fe A$, and this proves that $\V\fe\U$ for every ultrafilter $\V$. Hence $\U$ is maximal.\end{proof}

In bibliografia devo aggiungere un lavoro di Beiglbock ed uno di Krautzberger

\end{document}